\documentclass[12pt,a4paper]{article}

\usepackage[english]{babel}
\usepackage{amsfonts, amsbsy, amssymb,amsmath,amsthm,graphicx,array}
\usepackage{graphics}

\newtheorem{theorem}{Theorem}
\newtheorem{nn_theorem}{Theorem}
\newtheorem{lemma}{Lemma}
\newtheorem{proposition}{Proposition}

\theoremstyle{definition}
\newtheorem{definition}{Definition}
\theoremstyle{remark}
\newtheorem{remark}{Remark}
\newtheorem{nn_remark}{Remark}

\newcommand{\st}{\raisebox{-3pt}{\LARGE $*$}}
\newcommand{\wG}{\widetilde G}
\newcommand{\hG}{\widehat\Gamma}

\begin{document}

\markboth{Igor Nikonov}
{A new proof of Vassiliev's conjecture}


\title{A new proof of Vassiliev's conjecture}

\author{Igor Nikonov\\
%
Department of Mechanics and Mathematics,\\
Moscow State University, Russia\\
nikonov@mech.math.msu.su
}

 \maketitle

\abstract{
 In the paper~\cite{Vas} on finite type invariants of self-intersecting curves,
V.A. Vassiliev conjectured a criterion of planarity of framed four-valent
graphs, i.e. $4$-graphs with an opposite edge structure at each
vertex. The conjecture was proved by V.O. Manturov~\cite{Man}.
We give here another proof of Vassiliev's planarity criterion of
framed four-valent graphs (and more generally, (even) $\st$-graphs),
which is based on Pontryagin--Kuratowski theorem.
 }


\section{Introduction}

\begin{definition}\label{def:star_graph}
A {\em\st-graph} is a graph for which at each vertex an unoriented
cyclic order of the outgoing half-edges is given. The {\em
unoriented cyclic order} can be determined by a bijecton from the
half-edges to the vertices of a cycle graph. Half-edges which are
mapped to adjacent vertices are called {\em adjacent}.

We call a \st-graph {\em even} if each vertex of the graph has even degree.
\end{definition}

\begin{remark}
If all the vertices of a \st-graph are of order $4$ then we have a
framed four-valent graph considered in~\cite{Man,Vas}.
\end{remark}

\begin{definition}
An {\em embedding} of a \st-graph $G$ into a surface $S$ is an
embedding of $G$ (as an ordinary graph) into $S$ which is compatible
with \st-structure, i.e. for any vertex $v$ of $G$ the cyclic order
on the half-edges at $v$ induced by the embedding must coincide with
the cyclic order of the \st-graph at $v$.

A \st-graph is called {\em planar} if there is an embedding of it into the plane $\mathbb R^2$.
\end{definition}

\begin{definition}
Let $v$ be a vertex of a \st-graph $G$. Let $\gamma_1$ and
$\gamma_2$ be paths in $G$ which have no common edges. Assume that
$\gamma_1$ and $\gamma_2$ go through the vertex $v$, and $e, e'$
(correspondingly, $f,f'$) are the half-edges of the path $\gamma_1$
(correspondingly, $\gamma_2$) incident to $v$. We say that
$\gamma_1$ and $\gamma_2$ {\em intersect transversely} at $v$ if the
pairs of half-edges $e,e'$ and $f,f'$ alternate in the unoriented
cyclic order at $v$ (see fig.~\ref{fig:transversal_intersection}).

 \begin{figure}
  \centering
  \begin{tabular}{cc}
    \includegraphics[width=0.25\textwidth]{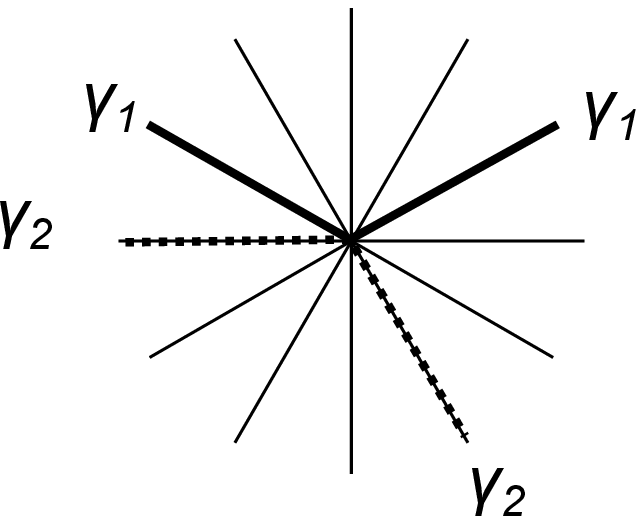} &
    \includegraphics[width=0.25\textwidth]{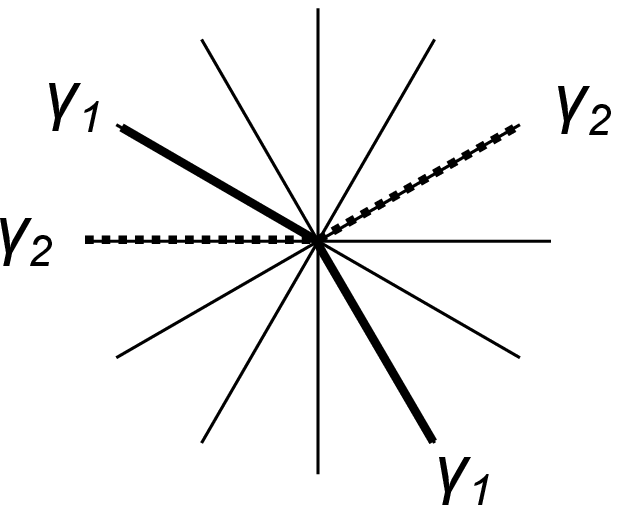} \\
    nontransversal intersection & transversal intersection
  \end{tabular}
  \caption{Intersection of paths at a vertex}\label{fig:transversal_intersection}
 \end{figure}

\end{definition}

\begin{remark}
The paths $\gamma_1$ and $\gamma_2$ can pass several times through $v$ and have several transversal intersection
at the vertex $v$. In order to distinguish intersection points we can consider the following construction. Let
$e_{2i-1},e_{2i},\ i=1,\dots,k,$ be the pairs of consecutive (half)edges of the paths $\gamma_1$ which are
incident to $v$. Analogously, we denote the edges of the path $\gamma_2$ incident to $v$ as $f_{2j-1},f_{2j},\
i=j,\dots,l$. The unoriented cyclic order at the vertex $v$ defines a bijection from half-edges to vertices of a
cyclic graph. Draw this cyclic graph on the plane as a circle with marked points on it. For any $i=1,\dots,k$
($j=,\dots,l$) connect the points that correspond to the edges $e_{2i-1}$ and $e_{2i}$ (edges $f_{2j-1}$ and
$f_{2j}$) with a line segment. We shall call the obtained graph as {\em vertex chord diagram} of $v$ (see
fig.~\ref{fig:vertex_chord_diagram}). Then the transversal intersection at the vertex correspond to
intersections of the chords. For example, the paths $\gamma_1$ and $\gamma_2$ in
fig.~\ref{fig:vertex_chord_diagram} have one transversal intersection (and there is a transversal
self-intersection of the path $\gamma_2$).

 \begin{figure}
  \centering\includegraphics[width=0.6\textwidth]{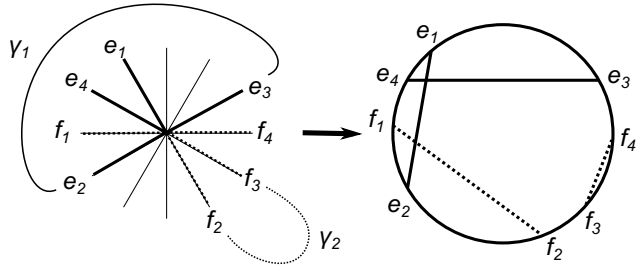}
  \caption{Vertex chord diagram}\label{fig:vertex_chord_diagram}
 \end{figure}

\end{remark}

Now we can formulate the main theorem of the paper.
\begin{theorem}[Vassiliev's planarity criterion for even \st-graphs]\label{thm:main_theorem}
An even $\ast$-graph is planar if and only if it does not contain a pair of
cycles without common edges and with exactly one transversal
intersection.
\end{theorem}

\begin{remark}
This criterion is not valid for non even \st-graphs, a counterexample is the graph $K_{3,3}$. This graph appears to be the only additional obstruction to planarity of an arbitrary \st-graph.
\end{remark}

\begin{theorem}[Planarity criterion for \st-graphs]\label{thm:main_theorem2}
An $\ast$-graph is not planar if and only if it contains a pair of cycles without common edges and with exactly one transversal intersection or contains a subgraph isomorphic to $K_{3,3}$.
\end{theorem}

\begin{remark}
Vassiliev's conjecture was originally formulated for framed 4-valent graphs~\cite{Vas}
and was proved under that restriction by V.O.~Manturov~\cite{Man}. Later T.~Friesen
generalized the result to \st-graphs with vertices of degree $4$ or $6$~\cite{Fri}. The
approach developed by Manturov was based on considering rotating Euler circuits of the
graph and allowed not only to prove the planarity criterion but also to give an
estimation of the genus of the graph (more accurately, the minimal genus of the surface
where the graph can be embedded so that its $\mathbb Z_2$-homology class is
trivial)~\cite{Man2,FriMan}.
\end{remark}

\begin{remark}
As I leaned from S.~Chmutov, an approach to Vassiliev's conjecture (for framed $4$-valent
graphs) similar to the one described here was found by A.~Kieger~\cite{Krieger} who
independently introduced web graphs and proved
Proposition~\ref{prop:planarity_equivalence_star_and web_graphs} below. However, he did
not finish consideration of all the cases needed for Theorem~\ref{thm:main_theorem}.
\end{remark}

\begin{remark}\label{rem:transersal_simple_cycles}
We can suppose that the transversal cycles in the theorem~\ref{thm:main_theorem} are simple. Indeed,let $C_1$ and $C_2$ be cycles with one transversal intersection. If $C_1$ or $C_2$ is not simple then it can be reduced to a simple one the following way. Assume that cycle $C_1$ goes several times through a vertex $v$ of the graph and consider the vertex chord diagram at $v$. Then the diagram contains chords that belong to the cycle $C_1$. There are chords $e_1e_2$ and $e_3e_4$ of $C_1$ such that the diagram circle does not contain between $e_2$ and $e_3$ any ends of chords of $C_1$. If there are no ends of chords between $e_2$ and $e_3$ (see fig.~\ref{fig:cycle_reductions} top) then we can split the cycle $C_1$ into two cycles $C'_1$ and $C''_1$ one of which has transversal intersection with $C_2$ and has less self-intersections than $C_1$. So assume that there are some chord ends of $C_2$ between $e_2$ and $e_3$. If those ends belong to different chords we can reduce the cycle $C_2$ by the reasoning above. If there is one or two end of a chord of $C_2$ between $e_2$ and $e_3$ then we can split the cycle $C_1$ and take the half with intersects transversely with $C_2$ (fig.~\ref{fig:cycle_reductions} middle and bottom). After all reductions we get two cycles which have one transversal intersection and each cycle can go through any vertex of the graph only once.

 \begin{figure}[h]\centering
 \includegraphics[width=0.5\textwidth]{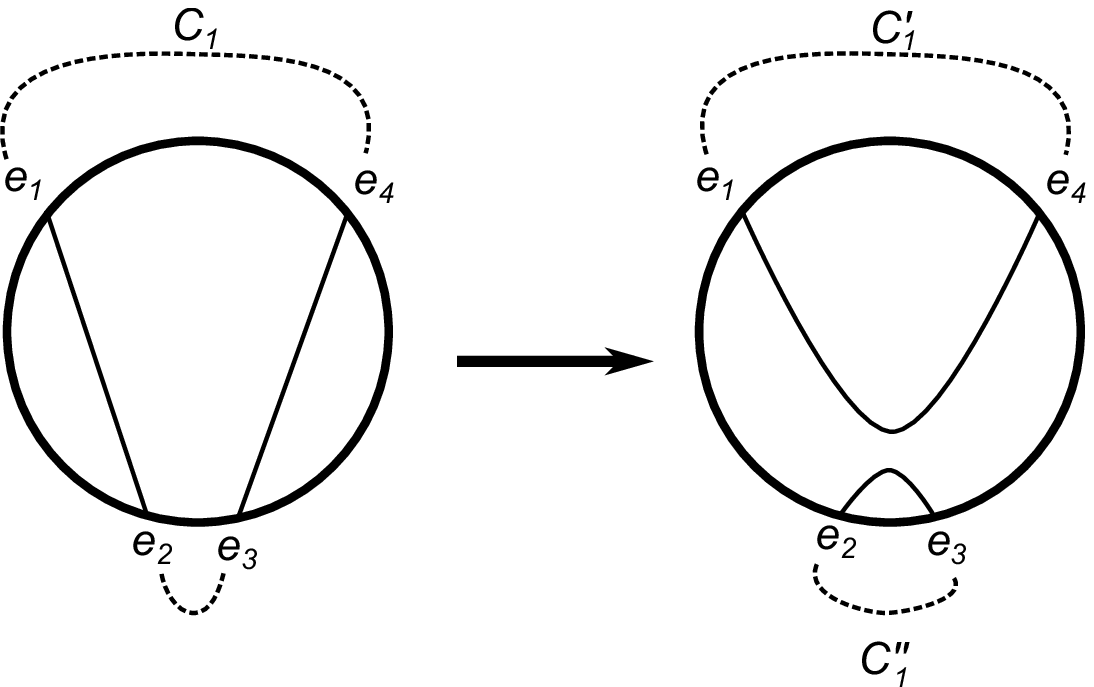}\\
 \includegraphics[width=0.5\textwidth]{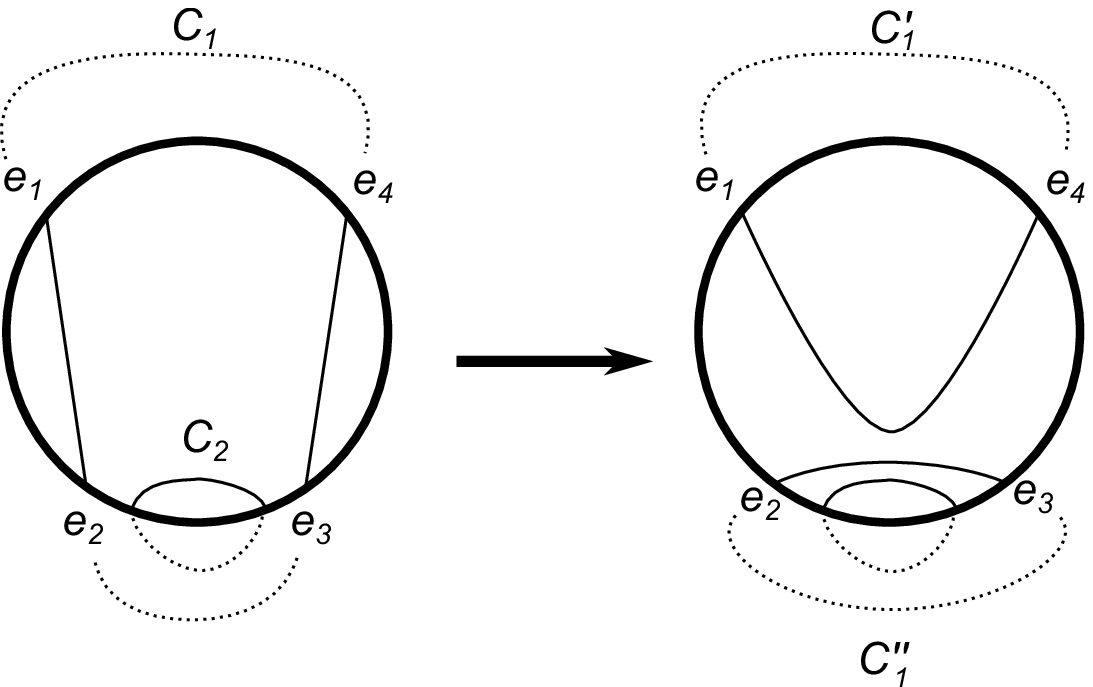}\\
 \includegraphics[width=0.5\textwidth]{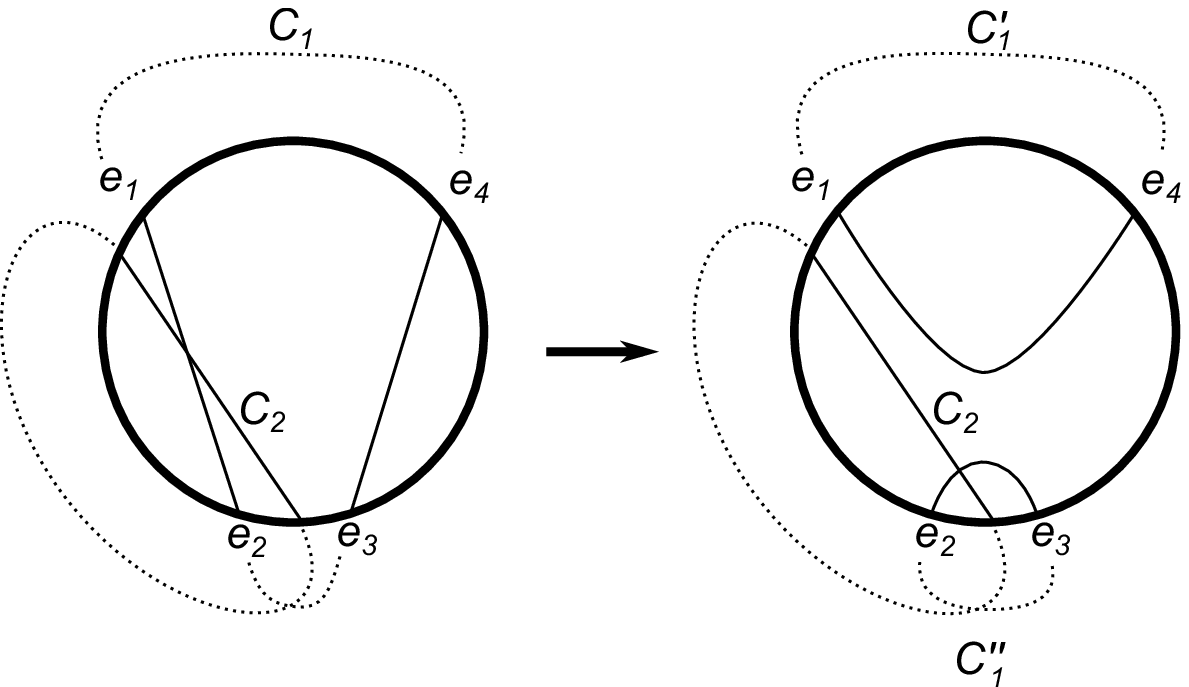}
  \caption{Reduction of transversal cycles}\label{fig:cycle_reductions}
 \end{figure}
\end{remark}

For ordinary graphs, a known planarity criterion is Pontryagin--Kuratowski theorem.

\begin{nn_theorem}[Pontryagin--Kuratowski planarity criterion]\nonumber
A graph is planar if and only if it contains a subgraph homeomorphic to $K_5$ or $K_{3,3}$
(fig.~\ref{fig:graphs_K5_K33}).
\end{nn_theorem}

\begin{figure}
  \centering
  \tabcolsep=2em
  \begin{tabular}{cc}
  \includegraphics[width=0.25\textwidth]{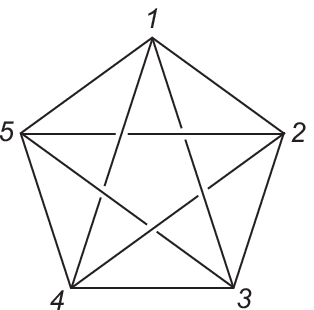}
  & \includegraphics[width=0.25\textwidth]{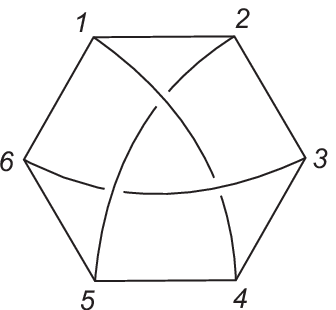} \\
  $K_5$ & $K_{3,3}$
  \end{tabular}
  \caption{Graphs $K_5$ and $K_{3,3}$}\label{fig:graphs_K5_K33}
 \end{figure}

\section{The proof}

In order to deduce Vassiliev's criterion from Pontryagin--Kuratowski planarity criterion one need to find a way to assign to a \st-graph an ordinary graph whose planarity is closely related to the planarity of the \st-graph.

\begin{definition}
Let be $G$ be a \st-graph. We construct the {\em web graph} $\wG$ of the graph $G$ by converting each vertex of $G$ into a {\em web}: on every half-edge incident to a given vertex of $G$ we set a new vertex and connect the new vertices of adjacent half-edges with an edge (see fig.~\ref{fig:node}). Informally speaking, we draw a circle around every vertex of $G$.

 \begin{figure}[h]
  \centering\includegraphics[width=0.25\textwidth]{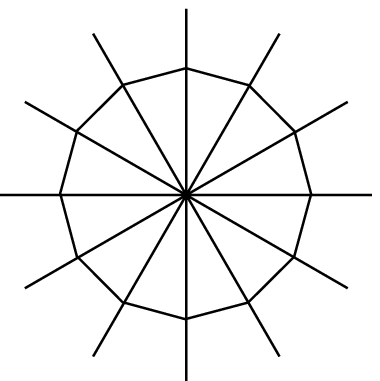}
  \caption{The web}\label{fig:node}
 \end{figure}

We call the outgoing half-edges of a web as {\em directions}. There is an unoriented order on the directions of
a web induced from the \st-structure of the graph.
\end{definition}

\begin{remark}
The web graph $\wG$ possesses a natural structure of a \st-graph but we shall consider $\wG$ only as an ordinary graph.
\end{remark}

The web graph $\wG$ contains a subgraph isomorphic to $G$ (one should remove the edges on circles of the webs
from $\wG$). On the other hand, there is a projection $\pi\colon\wG\to G$ which maps any web in $\wG$ to the
corresponding vertex of $G$.

The following technical result will be useful in constructing transversal cycles below.

\begin{lemma}\label{lem:transversal_path_projection}
Let $\gamma_1$ and $\gamma_2$ be paths in the web graph $\wG$ which don't intersect internally, i.e. $(\gamma_1\setminus\partial\gamma_1) \cap (\gamma_2\setminus\partial\gamma_2) = \emptyset$. Then the paths $\pi(\gamma_1)$ and $\pi(\gamma_2)$ in $G$ have no transversal intersections unless one of the paths becomes closed after projection and one of the following situations takes place:
\begin{itemize}
\item the ends of the path $\gamma_i$, $i=1$ or $2$, lie in one web $W$ and are separated by the path $\gamma_{3-i}$, i.e. the ends lie in different components of the graph $W\setminus\gamma_{3-i}$ (see fig.~\ref{fig:node_paths_transversal} left);
\item the ends of the paths $\gamma_1$ and $\gamma_2$ lie in one web $W$ and the corresponding directions alternate in the unoriented cyclic order of $W$ (see fig.~\ref{fig:node_paths_transversal} right).
\end{itemize}

 \begin{figure}[h]
  \centering\includegraphics[width=0.7\textwidth]{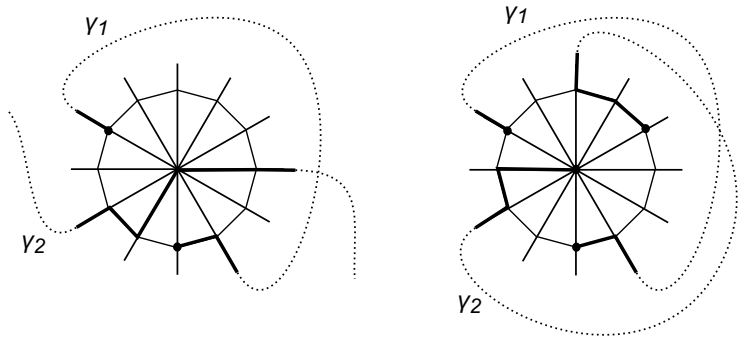}
  \caption{Configurations of paths that produce transversal intersections}\label{fig:node_paths_transversal}
 \end{figure}
\end{lemma}

\begin{proof}
Indeed, if $\gamma_1$ and $\gamma_2$ go through a web $W$ and their projections have a transversal intersection at the vertex $\pi(W)$ then the half-edges of $\pi(\gamma_1)$ and $\pi(\gamma_2)$ must alternate in the unoriented cyclic order at $\pi(W)$, so the directions of $\gamma_1$ and $\gamma_2$ at the web $W$ alternate too. Then the paths $\gamma_1$ and $\gamma_2$ must have an intersection point in $W$.

Thus, transversal intersection can only appear when some path becomes closed. Let the path $\gamma_1$ start and end in a web $W$ and $\gamma_2$ go through $W$. If the ends of $\gamma_1$ are separated by an arc in $\gamma_2\cap W$ then the initial and final directions of $\gamma_1$ at $W$ and the direction of the arc in $\gamma_2$ alternate in the cyclic order, so the corresponding edges of $\pi(\gamma_1)$ and $\pi(\gamma_2)$ alternate too and a transversal intersection appears. If the ends of $\gamma_1$ are not separated then we can embed the web $W$ in a disk $D_W\subset\mathbb R^2$ and connect the ends of $\gamma_1$ in $D_W$ with an ark $\delta\subset D_W\setminus(\gamma_2\cap W)$. Then $\delta$ and the initial and final part of $\gamma_1$ splits $D_W$ into two components and any arc $\gamma_2\cap W$ lies in one of these components. This ensures that the initial and final directions of $\gamma_1$ at $W$ and the direction of the arc in $\gamma_2$ don't alternate, hence the intersection $\pi(\gamma_1)$ and $\pi(\gamma_2)$ is not transversal.

The case when the ends of both the paths $\gamma_1$ and $\gamma_2$ lie in one web can be considered analogously.
\end{proof}

Consideration of the web graph is justified by the following result.

\begin{proposition}\label{prop:planarity_equivalence_star_and web_graphs}
A \st-graph $G$ is planar if and only if its web graph $\wG$ is planar (as an ordinary graph).
\end{proposition}

\begin{proof}
Indeed, if there is an embedding of $G$ into plane that conserves the unoriented cyclic order at each vertex then one can extend the embedding to an embedding of $\wG$ by drawing small circles around the images of the vertices of $G$.

On the other hand, if there is an embedding of $\wG$ we can take the restriction of the embedding to the subgraph of $\wG$ isomorphic to $G$. It will be an embedding of $G$ provided it preserves the cyclic order at vertices. It is suffice to show that the unoriented cyclic order does not change for any four directions of any web in $\wG$. So, assume that for some web with the central vertex $O$ and non-central vertices $A_1, A_2, A_3, A_4$ (arranged according the unoriented cyclic order) the order of the directions $OA_1$, $OA_2$, $OA_3$, $OA_4$ is changed by the embedding to $OA_1$, $OA_3$, $OA_2$, $OA_4$. Then we can construct an embedding of $K_5$ into plane as follows (see fig.~\ref{fig:cyclic_order_keeps}). Consider the subgraph of the embedded graph, which contains the edges $OA_i$, $i=1,2,3,4$ and the circle of the web. We take points $A'_i$, $i=1,2,3,4$, on the edges $OA_i$ near the point $O$. Since the cyclic order was changes we can add edges $A'_1A'_3$, $A'_3A'_2$, $A'_2A'_4$ and $A'_4A'_3$ to the subgraph. Then the new edges, edges $OA'_i$, $A'_iA_i$, $i=1,2,3,4$, and paths $A_1A_2$ and $A_3A_4$, which lies on the web circle, form an embedded graph isomorphic to $K_5$. The contradiction with Pontryagin--Kuratowski theorem implies the embedding must be compatible with the \st-structure.
\end{proof}

 \begin{figure}[h]
  \centering\includegraphics[width=0.3\textwidth]{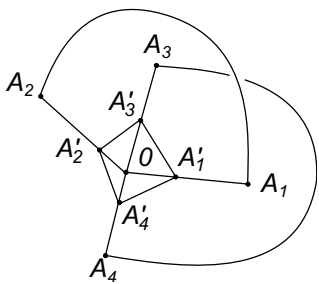}
  \caption{Embedding of $K_5$ induced by wrong cyclic order}\label{fig:cyclic_order_keeps}
 \end{figure}

\noindent {\bf Proof of Theorem~\ref{thm:main_theorem}.}
We start the proof of theorem~\ref{thm:main_theorem} with the simple part. Let $G$ be a \st-graph that contains a pair of cycles $C_1$ and $C_2$ which has exactly one transversal intersection. According remark~\ref{rem:transersal_simple_cycles}, we can suppose $C_1$ and $C_2$ to be simple cycles. Since $G$ embeds in its web graph $\wG$, the cycles $C_1$ and $C_2$ can be considered as cycles in $\wG$. We modify the cycles as following: we separate the cycles in the webs where they have nontransversal intersections and add the circle of the web where the cycles have transversal intersection (see fig.~\ref{fig:necessity}). The resulting subgraph will be isomorphic to $K_5$ so the graph $\wG$ as well as the \st-graph $G$ will be nonplanar.

 \begin{figure}[h]
  \centering\includegraphics[width=0.6\textwidth]{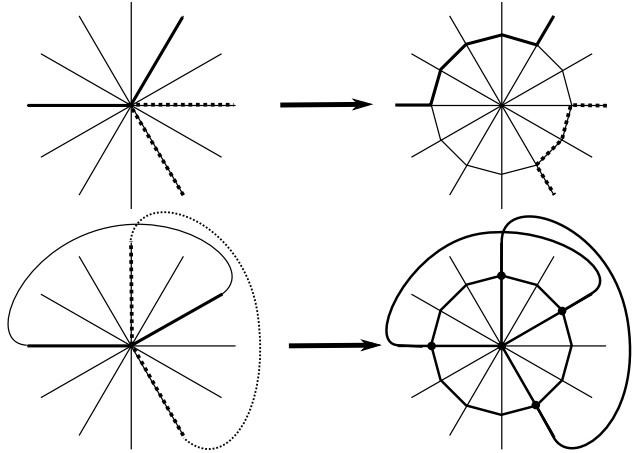}
  \caption{Construction of $K_5$ from a pair of transversal cycles}\label{fig:necessity}
 \end{figure}

Now let $G$ be a nonplanar even \st-graph. Then its web graph $\wG$ is nonplanar too. By Pontryagin--Kuratowski
theorem $\wG$ contains a subgraph $\Gamma$ isomorphic to $K_5$ and $K_{3,3}$. Later we shall ignore the vertices
of $\Gamma$ of degree $2$ and assume that a {\em vertex} of $\Gamma$ has always degree $3$ (when $\Gamma\simeq
K_{3,3}$) or $4$ (when $\Gamma\simeq K_5$) and an {\em edge} of $\Gamma$ is in fact a path in $\Gamma$ whose
inner vertices are of degree two.

\begin{remark}\label{rem:node_disk}
 Let $W$ be a web in $\wG$. Since the graph $W$ is planar we can embed it and its outgoing half-edges into
a disk $D_W\subset\mathbb R^2$. Denote the part of the part of $\Gamma$ which lie inside $D_W$ as $\Gamma_W$.
This part can be considered as a plane graph with internal vertices of degree $3$ or $4$ and boundary vertices
on degree $1$ which lie on the boundary of the disk $D_W$ (see fig.~\ref{fig:node_disk}). Let $\bar\Gamma_W$ be
the union of components of $\Gamma_W$ which contain internal vertices and $\bar\Gamma^{int}_W$ be the full
subgraph of $\bar\Gamma_W$ spanned on its internal vertices. Remark that the graph $\bar\Gamma_W$ is determined
up to isomorphism by $\bar\Gamma^{int}_W$.

Let $v_1, v_2$ be adjacent vertices of $\Gamma$ such that $v_1\in W$ and the edge $v_1v_2\subset\Gamma$ does not
lie in $W$. We define the {\em direction $v_1\to v_2$ of the edge $v_1v_2$} as the outgoing half-edge of $W$
which appears first in the edge $v_1v_2$ (i.e. the path in $\Gamma$ that connects $v_1$ and $v_2$).

 \begin{figure}
  \centering\includegraphics[width=0.4\textwidth]{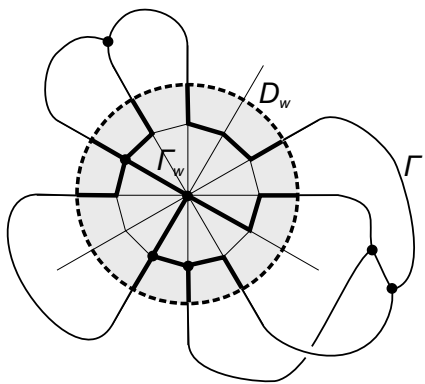}
  \caption{The web disk}\label{fig:node_disk}
 \end{figure}
\end{remark}

The rest of the proof of the main theorem is essentially a search through possible configurations for plane
graphs $\Gamma_W$.

Case $K_5$. The subgraph $\Gamma$ is isomorphic to $K_5$.\\
Denote the vertices of $\Gamma$ as $1,2,3,4,5$.

Case  $K_5$.a. There is a web $W$ that contains only one vertex of $\Gamma$.\\
Without loss of generality, we can assume that $1\in W$ and that directions of the edges incident to $1$ go in
the following cyclic order at $W$: $1\to 2, 1\to 3, 1\to 4, 1\to 5$. Consider the paths $\gamma_1=1241$ and
$\gamma_2=1351$ in $\Gamma$ (see. fig.~\ref{fig:case_K5_v1}). By lemma~\ref{lem:transversal_path_projection} the
projections $\pi(\gamma_1)$ and $\pi(\gamma_2)$ form a pair of cycles with one transversal intersection.

 \begin{figure}[h]
  \centering\includegraphics[width=0.3\textwidth]{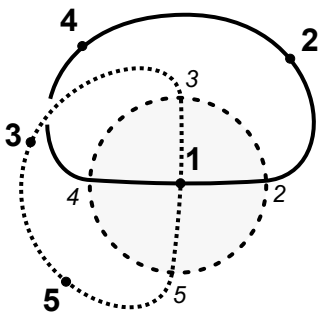}
  \caption{Case $K_5$.a : one vertex in the web}\label{fig:case_K5_v1}
 \end{figure}

Case  $K_5$.b. There is a web $W$ that contains exactly two vertices of $\Gamma$. \\
We can assume that $W$ contains the vertices $1$ and $2$. Then vertices $1$ and $2$ are connected inside $W$
(i.e. $12\subset W$) through the central vertex of $W$. Let $i,j,k$ (corr., $k',j',i'$) be the numbers of
vertices given in the cyclic order of the directions of edges connecting them to the vertex $1$ (corr., $2$),
see fig.~\ref{fig:case_K5_v2}. Then $\{i,j,k\}=\{i',j',k'\}=\{3,4,5\}$.

 \begin{figure}[h]
  \centering\includegraphics[width=0.25\textwidth]{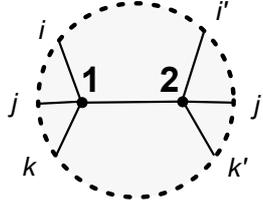}
  \caption{Case $K_5$.b : two vertices in the web}\label{fig:case_K5_v2}
 \end{figure}

Case $K_5$.b.1. The orders of edges incident to the vertices $1$ and $2$ are not compatible
($(i,j,k)\ne(i',j',k')$).\\
Let $i\ne i'$. We can suppose that $i=3,i'=5$. Then by lemma~\ref{lem:transversal_path_projection} the
projections of paths $\gamma_1=132$ and $\gamma_2=152$ give a Vassiliev's obstruction in $G$ (see
fig.~\ref{fig:case_K5_v2_subcases}). The case $k\ne k'$ is considered analogously.

  \begin{figure}[!ht]
  \centering\includegraphics[width=0.3\textwidth]{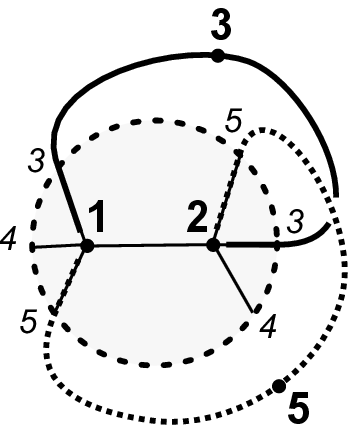}
  \qquad \includegraphics[width=0.3\textwidth]{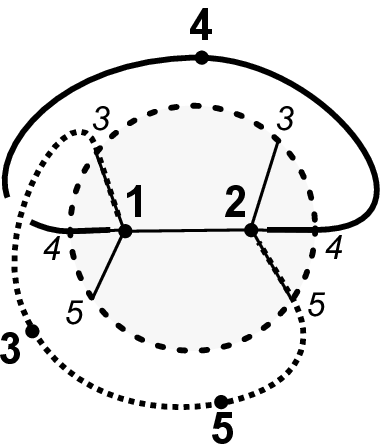}
  \caption{Cases $K_5$.b.1 (left) and $K_5$.b.2 (right)}\label{fig:case_K5_v2_subcases}
 \end{figure}

Case $K_5$.b.2. The orders of edges incident to the vertices $1$ and $2$ are compatible
($(i,j,k)=(i',j',k')$).\\
We can assume that $i=3,j=4,k=5$. Then paths $\gamma_1=142$ and $\gamma_2=1352$ give a Vassiliev's obstruction
(see fig.~\ref{fig:case_K5_v2_subcases}).

Case $K_5$.c. There is no webs which contains one or two vertices of $\Gamma$.\\
Then all the vertices lie in the same web $W$. Since for any vertex of $\Gamma$ one of edges incident to it goes
through the center of the web, the central vertex of $W$ is a vertex of $\Gamma$ and the other four vertices of
$\Gamma$ lie in the web circle. Enumerate the vertices of $\Gamma$ so that $1$ be the central and the vertices
$2,3,4,5$ follow in the given cyclic order (see fig.~\ref{fig:case_K5_v5}). Then the vertices $2$ and $4$ ($3$
and $5$) can not be connected inside the web, so the directions $2\to 4$, $3\to 5$, $4\to 2$ and $5\to 3$ are
defined. Moreover, the directions are cyclically arranged as written above. Hence, paths $\gamma_1=1241$ and
$\gamma_2=1351$ give a Vassiliev's obstruction.

 \begin{figure}[!ht]
  \centering\includegraphics[width=0.3\textwidth]{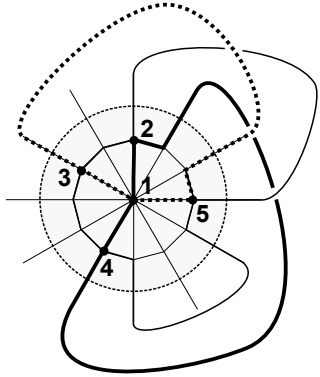}
  \caption{Case $K_5$.c : five vertices in the web}\label{fig:case_K5_v5}
 \end{figure}

Case $K_{3,3}$. The subgraph $\Gamma$ is isomorphic to $K_{3,3}$. \\
 We denote the vertices of $\Gamma$ as $1,2,3,4,5,6$ so that the vertices of different parity are adjacent.

Case $K_{3,3}$.a. There is a web $W$ that contains exactly two vertices of $\Gamma$.

Case $K_{3,3}$.a.1. The two vertices in $W$ are connected in $W$.\\
We can assume that $1,2\in W$. Enumerate the other vertices so that the directions of edges incident to the
vertices  $1$ and $2$ follow in the cyclic order $1\to 4, 1\to 6, 2\to 5, 2\to 3$ (see
fig.~\ref{fig:case_K33_v2_connected}). Then paths $\gamma_1=1452$ and $\gamma_2=1632$ give a Vassiliev's
obstruction.

 \begin{figure}[!ht]
  \centering\includegraphics[width=0.3\textwidth]{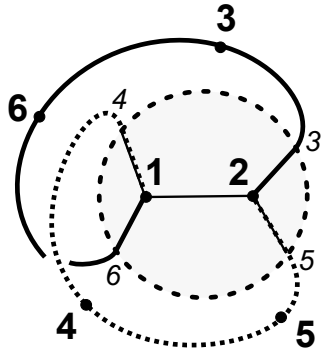}
  \caption{Case $K_{3,3}$.a.1 : two connected vertices in the web}\label{fig:case_K33_v2_connected}
 \end{figure}

Case $K_{3,3}$.a.2. The two vertices in $W$ are not separated, i.e. they can be connected in $D_W$
with an arc which does not intersect $\Gamma$.\\
There are two possibilities: the vertices in $W$ are either adjacent in $\Gamma$ or not (see
fig.~\ref{fig:case_K33_v2_nonseparated}). In the first case, we can use the reasonings of the case
$K_{3,3}$.a.1; in the last case, we can follow the reasonings of the case $K_5$.b.

  \begin{figure}[!ht]
  \centering\includegraphics[width=0.25\textwidth]{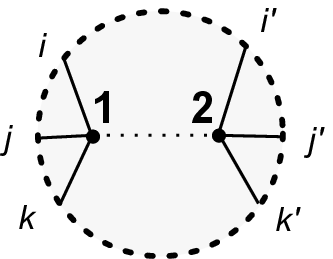}
  \qquad \includegraphics[width=0.25\textwidth]{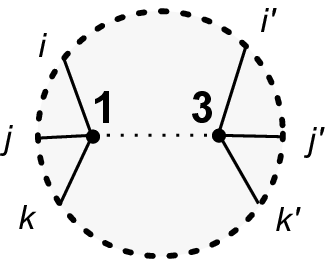}
  \caption{Case $K_{3,3}$.a.2 : two nonconnected nonseparated vertices in the web }\label{fig:case_K33_v2_nonseparated}
 \end{figure}

Case $K_{3,3}$.a.3. The two vertices in $W$ are separated by an arc
of $\Gamma\cap W$.

\begin{nn_remark}
There can be only one separating arc in a web because the arc must
contain the central vertex of the web.
\end{nn_remark}

 We can suppose that $W$ contains the vertices
$1$ and $2$ (if the vertices are adjacent in $\Gamma$) or $1$ and
$3$ (if they are not adjacent). Denote by $i,j,k,k',j',i'$ the
number of vertices arranged according the cyclic order of directions
of edges incident to $1$ or $2$ (or $1$ or $3$), and let the
separating arc belong to the edge $pq$ (see
fig.~\ref{fig:case_K33_v2_separated}).
  \begin{figure}[!ht]
  \centering\includegraphics[width=0.25\textwidth]{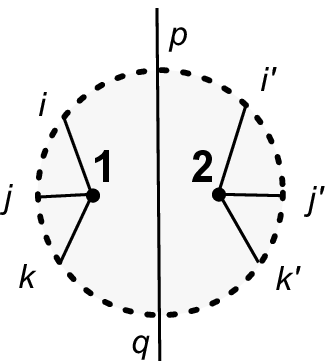}
  \qquad \includegraphics[width=0.25\textwidth]{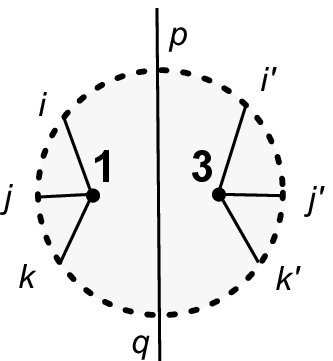}
  \caption{Cases $K_{3,3}$.a.3.1 (left) and $K_{3,3}$.a.3.2 (right): two separated vertices in the web }\label{fig:case_K33_v2_separated}
 \end{figure}

Case $K_{3,3}$.a.3.1. The two separated vertices in $W$ are adjacent.\\
If $\{p,q\}=\{1,2\}$ then we can substitute the edge $12$ with an
arc, which connects $1$ and $2$ inside $D_W$, and apply the
reasonings of the case $K_{3,3}$.a.1. If $\{p,q\}\cap
\{1,2\}=\emptyset$, for instance, $p=3$, $q=4$ then the paths
$\gamma_1=12$ and $\gamma_2=34563$ yield a Vassiliev's obstruction
(see fig.~\ref{fig:case_K33_v2_(12)_separated} left). Finally,
$\left|\{p,q\}\cap \{1,2\}\right|=1$, for instance, $p=1$, $q=4$
then the paths we should choose depend on the position of the
direction $1\to 2$. If the direction $1\to 2$ is not separated in
the cyclic order from directions $2\to i$ by the directions $1\to 4$
and $1\to 6$ (see fig.~\ref{fig:case_K33_v2_(12)_separated} middle)
then we take $\gamma_1=12$ and $\gamma_2=14361$. If the direction
$1\to 2$ lies in the cyclic order between $1\to 4$ and $1\to 6$ (see
fig.~\ref{fig:case_K33_v2_(12)_separated} right) then we can assume
that the directions $2\to 1$ and $2\to 3$ are adjacent in the cyclic
order (otherwise we interchange the numbers of the vertices $3$ and
$5$) and take $\gamma_1=12341$ and $\gamma_2=1652$. In either case
the projections of the paths $\gamma_1$ and $\gamma_2$ form a pair
of transversal cycles.

  \begin{figure}[!ht]
  \centering\includegraphics[width=0.25\textwidth]{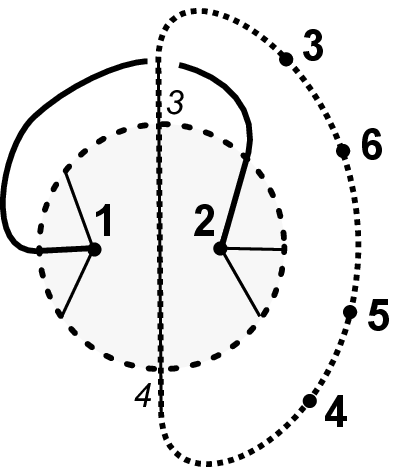}
  \quad \includegraphics[width=0.2\textwidth]{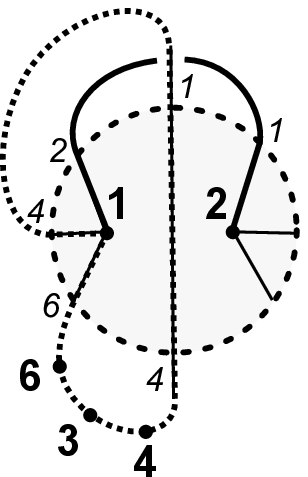}
  \quad \includegraphics[width=0.25\textwidth]{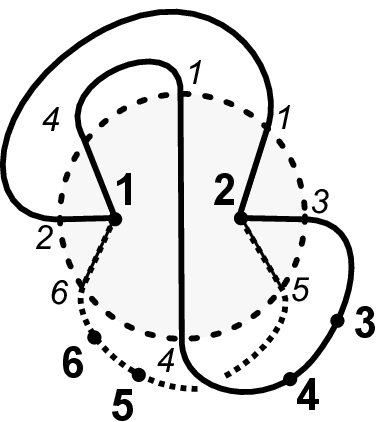}
  \caption{Case $K_{3,3}$.a.3.1: two separated adjacent vertices in the web }\label{fig:case_K33_v2_(12)_separated}
 \end{figure}

Case $K_{3,3}$.a.3.1. The two separated vertices in $W$ are not
adjacent.\\
Let $\{p,q\}\cap\{1,3\}=\emptyset$, for instance, $p=2$, $q=5$. We
can suppose that the directions $1\to 2$ and $1\to 4$ are adjacent
in the cyclic order (otherwise we interchange the numbers of the
vertices $4$ and $6$), see fig.~\ref{fig:case_K33_v2_(13)_separated}
left. Then the paths $\gamma_1=12541$ and $\gamma_2=163$ gives a
Vassiliev's obstruction after projection to $G$. Another possibility
is that $p$ or $q$ is equal $1$ or $3$, for example, $p=1$, $q=2$.
We can assume again that the directions $1\to 2$ and $1\to 6$ are
adjacent (see fig.~\ref{fig:case_K33_v2_(13)_separated} right). Then
we should take again the paths $\gamma_1=12541$ and $\gamma_2=163$
and get a Vassiliev's obstruction from them.

   \begin{figure}[!ht]
  \centering\includegraphics[width=0.25\textwidth]{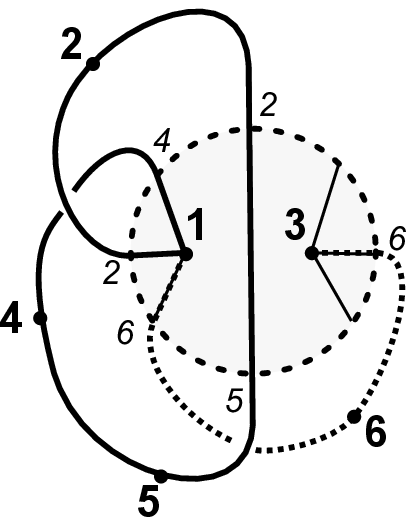}
  \qquad \includegraphics[width=0.25\textwidth]{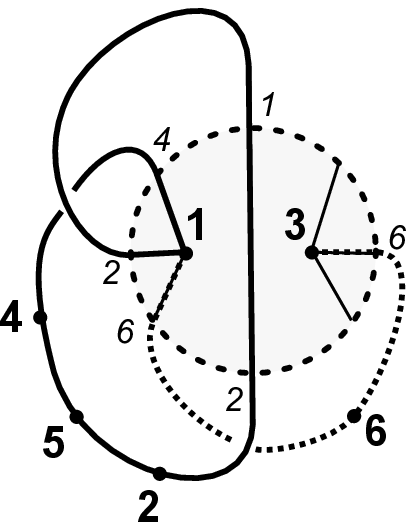}
  \caption{Case $K_{3,3}$.a.3.2 : two separated nonadjacent vertices in the web }\label{fig:case_K33_v2_(13)_separated}
 \end{figure}

Thus, the case $K_{3,3}$.a is completely investigated.

\begin{nn_remark}
If the number of vertices of $\Gamma$ in a web $W$ is grater than
$2$ we can use the following reductions.

If there is a separating arc then it splits the disc $D_W$ into two halves. Denote $D'$ to be the half which contains not less than two vertices of $\Gamma$ (see fig.~\ref{fig:node_reduction} left). Then $D'$ contains less or equal number of vertices of $\Gamma$ than the disc $D_W$. Moreover, the cyclic order of directions of edges outgoing from $D'$ coincides with the cyclic order in $D_W$. Hence if we construct a pair of paths $\gamma_1$ and $\gamma_2$ whose directions alternate in the cyclic order of $D'$ (as we did above) then the same will be true in the cyclic order $D_W$. Hence, the projections of the paths $\gamma_1$ and $\gamma_2$ in $G$ will be a pair of transversal cycles.

If the graph $\bar\Gamma_W$ (or $\bar\Gamma^{int}_W$) is not connected then the disc $D_W$ can be separated with an arc so that one of the halves $D'$ of the disc contains less vertices of $\Gamma$ than $D_W$ but contains at least two vertices. The reasonings of the previous paragraph show that we can construct the paths $\gamma_1$ and $\gamma_2$ considering the part of $\Gamma$ which lies in $D'$.

 \begin{figure}[!ht]
  \centering\includegraphics[width=0.20\textwidth]{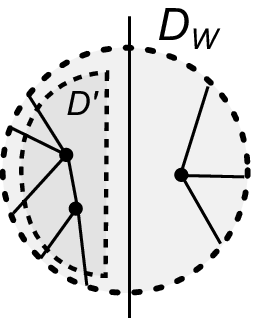}
  \qquad \includegraphics[width=0.2\textwidth]{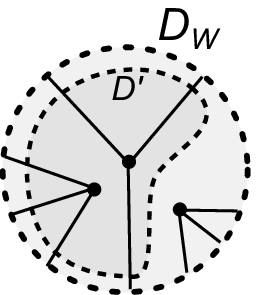}  \caption{Reductions for a separated (left) and nonconnected (right) cases}\label{fig:node_reduction}
 \end{figure}

Thus, later on we need to consider only nonseparated cases with
connected graphs $\bar\Gamma^{int}_W$. The graph
$\bar\Gamma^{int}_W$ must be a planar subgraph of $K_{3,3}$.
\end{nn_remark}

Case $K_{3,3}$.b. There is a web $W$ which contains three vertices of $\Gamma$.\\
Up to isomorphism there is only one connected subgraph in $K_{3,3}$ with three vertices. So let the web contain the vertices $1$, $2$ and $3$ and the vertices be connected inside $W$. Without loss of generality we can suppose that the directions of edges incident to the vertices $1$, $2$ and $3$ follow in the cyclic order $1\to 4$, $1\to 6$, $2\to 5$ then directions of the vertex $3$ (see fig.~\ref{fig:case_K33_v3}). Then the paths $\gamma_1=1452$ and $\gamma_2=163$ give a Vassiliev's obstruction.

 \begin{figure}[!ht]
  \centering\includegraphics[width=0.3\textwidth]{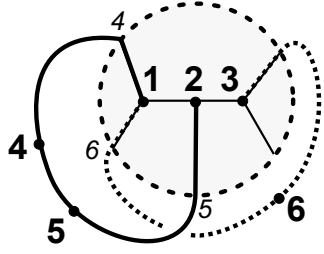}
  \caption{Case $K_{3,3}$.b : three vertices in the web}\label{fig:case_K33_v3}
 \end{figure}

Case $K_{3,3}$.c. There is a web $W$ which contains four vertices of $\Gamma$.\\
Up to isomorphism there are three connected subgraphs in $K_{3,3}$ with four vertices (see fig.~\ref{fig:node_graph_K33_v4}).

  \begin{figure}[!ht]
  \centering
  \begin{tabular}{ccc}
  \includegraphics[width=0.17\textwidth]{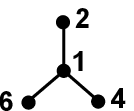} &
  \includegraphics[width=0.25\textwidth]{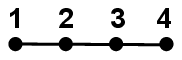} &
  \includegraphics[width=0.165\textwidth]{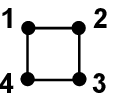}\\
  1) & 2) & 3)
  \end{tabular}
  \caption{Connected subgraphs of $K_{3,3}$ with $4$ vertices}\label{fig:node_graph_K33_v4}
 \end{figure}

Case $K_{3,3}$.c.1. The graph $\bar\Gamma^{int}_W$ is a tree with a vertex of degree $3$. \\
Then there are two vertices among $2,4,6$, for instance $2$ and $4$, such that the directions of the edges incident to these vertices are arranged in the unoriented cyclic order as follows: $2\to 3$, $2\to 5$, $4\to 3$, $4\to 5$ (see fig.~\ref{fig:case_K33_v4_1}). Then the projections of the paths $\gamma_1=234$ and $\gamma_2=254$ form a Vassiliev's obstruction.

Case $K_{3,3}$.c.2. The graph $\bar\Gamma^{int}_W$ is a tree without vertices of degree $3$.\\
We suppose that the vertices are enumerated as shown in fig.~\ref{fig:node_graph_K33_v4}.2). There are two possibilities: the directions of the edges $25$ and $36$ can be either separated by the directions of edges incident to the vertices $1$ and $4$ or not. If the directions $2\to 5$ and $3\to 6$ are not separated (see fig.~\ref{fig:case_K33_v4_2} left) then we choose the paths $\gamma_1=163$ and $\gamma_2=254$. If they are separated (see fig.~\ref{fig:case_K33_v4_2} right) we take the paths $\gamma_1=14$ and $\gamma_2=2563$. In both cases the paths $\gamma_1$ and $\gamma_2$ yield a pair of transversal cycles in $G$.

 \begin{figure}[!ht]
  \centering\includegraphics[width=0.25\textwidth]{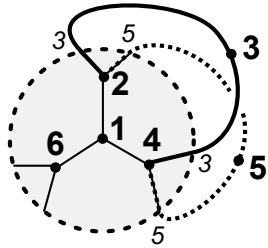}
  \caption{Case $K_{3,3}$.c.1}\label{fig:case_K33_v4_1}
 \end{figure}

   \begin{figure}[!ht]
  \centering\includegraphics[width=0.23\textwidth]{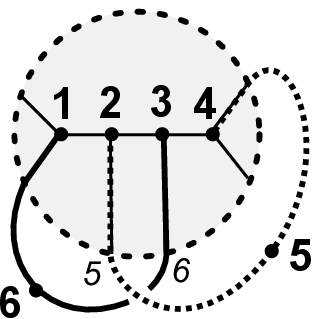}
  \qquad \includegraphics[width=0.28\textwidth]{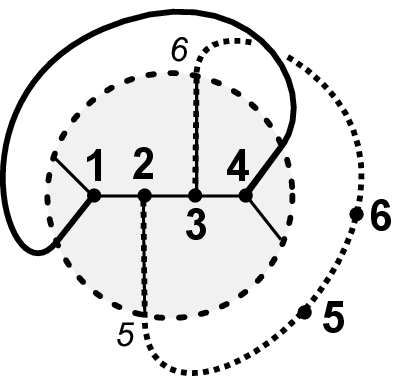}
  \caption{Case $K_{3,3}$.c.2}\label{fig:case_K33_v4_2}
 \end{figure}

Case $K_{3,3}$.c.3. The graph $\bar\Gamma^{int}_W$ is a cycle.\\
A Vassiliev's obstruction is obtained if we take the paths $\gamma_1=163$ and $\gamma_2=254$ (see fig.~\ref{fig:case_K33_v4_3}).

 \begin{figure}[!ht]
  \centering\includegraphics[width=0.25\textwidth]{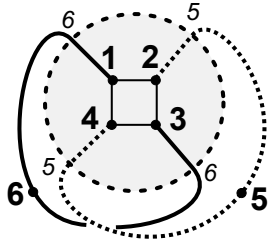}
  \caption{Case $K_{3,3}$.c.3}\label{fig:case_K33_v4_3}
 \end{figure}

There is a web $W$ which contains five vertices of $\Gamma$.\\

\begin{nn_remark}
The graph  $\bar\Gamma^{int}_W$ can not contain the graph $K_{2,3}$ (see fig.~\ref{fig:node_graph_K33_forbidden}). Indeed, if the graph $K_{2,3}\subset \bar\Gamma^{int}_W\subset\Gamma$ is embedded into the web $W$ then one of the vertices of $K_{2,3}$ of degree $2$ lies in the center of the web, whereas the edges between the other vertices $K_{2,3}$ cover the circle of the web. Hence the vertex of $\Gamma\setminus K_{2,3}$ lies outside the web $W$ and the central vertex can not be connected with it by a path which does not intersect the other edges of the graph $\Gamma$.

 \begin{figure}[!ht]
  \centering\includegraphics[width=0.13\textwidth]{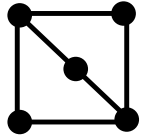}
  \caption{Forbidden subgraph of $K_{3,3}$}\label{fig:node_graph_K33_forbidden}
 \end{figure}
\end{nn_remark}

Excluding the forbidden graph, there are three connected subgraphs in $K_{3,3}$ with five vertices (see fig.~\ref{fig:node_graph_K33_v5}) up to isomorphism.

  \begin{figure}
  \centering
  \begin{tabular}{m{0.25\textwidth}m{0.21\textwidth}m{0.15\textwidth}}
  \includegraphics[width=0.25\textwidth]{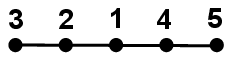} &
  \includegraphics[width=0.21\textwidth]{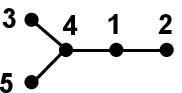} &
  \includegraphics[width=0.15\textwidth]{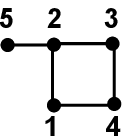}\\
  \multicolumn{1}{c}{1)} & \multicolumn{1}{c}{2)} & \multicolumn{1}{c}{3)}
  \end{tabular}
  \caption{Admissible connected subgraphs of $K_{3,3}$ with $5$ vertices}\label{fig:node_graph_K33_v5}
 \end{figure}

Case $K_{3,3}$.d.1. The graph $\bar\Gamma^{int}_W$ is a tree without vertices of degree $3$.\\
Enumerate the vertices of $\bar\Gamma^{int}_W$ as shown in fig.~\ref{fig:node_graph_K33_v5}.1).

Assume first, that the directions of the edges incident to the vertices $3$ and $5$ go in the cyclic  order  as follows: $3\to 4$, $5\to 2$, $5\to 6$, $3\to 6$. If the direction $2\to 5$ lies in the cyclic order between $3\to 6$ and $5\to 6$ then the paths $\gamma_1=3655$ and $\gamma_2=25$ give a Vassiliev's obstruction (see fig.~\ref{fig:case_K33_v5_1} left). The same reasonings work if the direction $4\to 3$ lies between $3\to 6$ and $5\to 6$. Hence, we can suppose that the directions $2\to 5$ and $4\to 3$ lies between $3\to 4$ and $5\to 2$ (see fig.~\ref{fig:case_K33_v5_1} middle). Then the projections of the paths $\gamma_1=25$ and $\gamma_2=34$ will have one transversal intersection.

Let the cyclic order of the directions of the edges incident to the vertices $3$ and $5$ be $3\to 4$, $3\to 6$, $5\to 2$, $5\to 6$. Repeating the reasonings of the previous paragraph we can construct a Vassiliev's obstruction if the direction $2\to5$ lies between $3\to 4$ and $5\to 6$ or the the direction $4\to3$ lies between $3\to 6$ and $5\to 2$. So we assume that $2\to5$ lies between $3\to 6$ and $5\to 2$ and $4\to3$ lies between $3\to 4$ and $5\to 6$. Without loss of generality we can suppose that the direction $1\to 6$ lies between $3\to 6$ and $5\to 2$ (see fig.~\ref{fig:case_K33_v5_1} right). Then a Vassiliev's obstruction is given by the paths $165$ and $25$.

 \begin{figure}[!ht]
  \centering\includegraphics[width=0.26\textwidth]{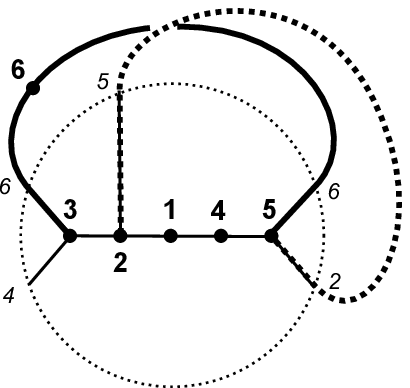}
  \qquad \includegraphics[width=0.22\textwidth]{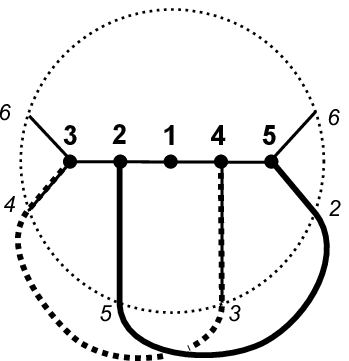}
  \qquad \includegraphics[width=0.25\textwidth]{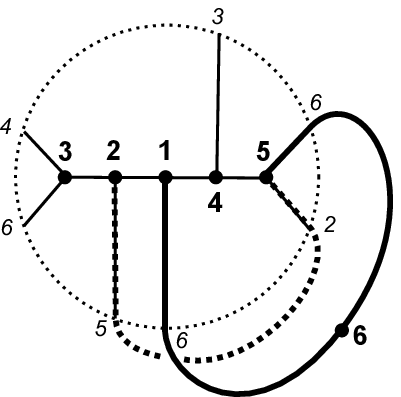}
  \caption{Case $K_{3,3}$.d.1}\label{fig:case_K33_v5_1}
 \end{figure}

Case $K_{3,3}$.d.2. The graph $\bar\Gamma^{int}_W$ is a tree with a vertex of degree $3$.\\
Enumerate the vertices of $\bar\Gamma^{int}_W$ as shown in fig.~\ref{fig:node_graph_K33_v5}.2). Without loss of generality we can suppose that the direction $1\to 6$ lies between the directions of the vertex $2$ and the directions of the vertex $3$.

Let the direction $3\to 6$ (and directions of the vertex $2$) separate  the direction $1\to 6$ from the direction $3\to 2$ (see fig.~\ref{fig:case_K33_v5_2} left). Then a Vassiliev's obstruction can be obtained from the paths $365$ and $25$.

Thus, we can assume that the direction $3\to 2$ lies between $3\to 6$ and $1\to6$. Then the paths $163$ and $23$ give a Vassiliev's obstruction.

 \begin{figure}[!ht]
  \centering\includegraphics[width=0.27\textwidth]{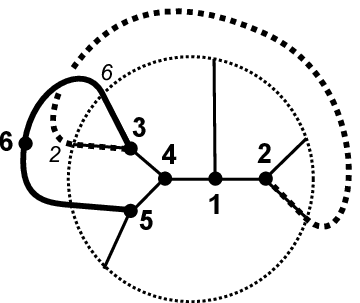}
  \qquad \includegraphics[width=0.21\textwidth]{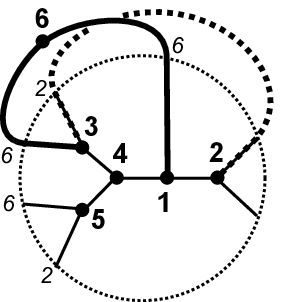}
  \caption{Case $K_{3,3}$.d.2}\label{fig:case_K33_v5_2}
 \end{figure}

Case $K_{3,3}$.d.3. The graph $\bar\Gamma^{int}_W$ contains a cycle, see fig.~\ref{fig:node_graph_K33_v5}.3).\\
In this case we should take the projections of the paths $163$ and $45$ (see fig.~\ref{fig:case_K33_v5_3}).

 \begin{figure}[!ht]
  \centering\includegraphics[width=0.25\textwidth]{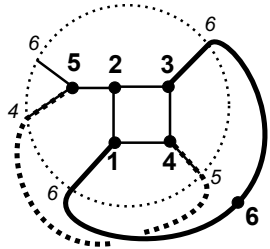}
  \caption{Case $K_{3,3}$.d.3}\label{fig:case_K33_v5_3}
 \end{figure}

Case $K_{3,3}$.e. The graph $\bar\Gamma^{int}_W$ contains $6$ vertices.\\
There are six connected subgraphs in $K_{3,3}$ with $6$ vertices which don't contain the forbidden graph (see fig.~\ref{fig:node_graph_K33_v6})

  \begin{figure}
  \centering
  \begin{tabular}{ccc}
  \includegraphics[width=0.25\textwidth]{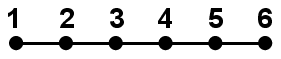} &
  \includegraphics[width=0.15\textwidth]{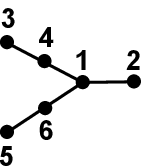} &
  \includegraphics[width=0.15\textwidth]{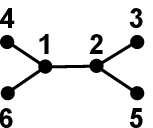}\\
  1) & 2) & 3)\\
  \includegraphics[width=0.15\textwidth]{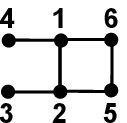} &
  \includegraphics[width=0.15\textwidth]{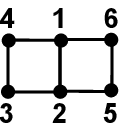} &
  \includegraphics[width=0.15\textwidth]{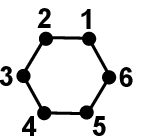}\\
  4) & 5) & 6)
  \end{tabular}
  \caption{Admissible connected subgraphs of $K_{3,3}$ with $6$ vertices}\label{fig:node_graph_K33_v6}
 \end{figure}

Case $K_{3,3}$.e.1. The graph $\bar\Gamma^{int}_W$ is a tree without vertices of degree $3$.\\
Enumerate the vertices of $\bar\Gamma^{int}_W$ as shown in fig.~\ref{fig:node_graph_K33_v6}.1).

Assume first, that the directions of the edges incident to the vertices $1$ and $6$ go in the cyclic  order  as follows: $1\to 4$, $1\to 6$, $6\to 3$, $6\to 1$. If the direction $4\to 1$ lies in the cyclic order between $1\to 6$ and $6\to 3$ then the paths $14$ and $16$ give a Vassiliev's obstruction (see fig.~\ref{fig:case_K33_v6_1}.1). The same reasonings work if the direction $3\to 6$ lies between $1\to 4$ and $6\to 1$. Hence, we can suppose that the direction $4\to 1$ lies between $1\to 4$ and $6\to 1$ and the direction $3\to 6$ lies between $1\to 6$ and $6\to 3$. Look at the position of the direction $2\to 5$. If it lies between $1\to 4$ and $6\to 1$ (see fig.~\ref{fig:case_K33_v6_1}.2) then we take the paths $14$ and $25$.  If the direction $25$ lies between $1\to 6$ and $6\to 3$ (see fig.~\ref{fig:case_K33_v6_1}.3) then we take the paths $163$ and $25$. The projections of the chosen paths form a Vassiliev's obstruction in the graph $G$.

So, let the cyclic order of the directions of the edges incident to the vertices $1$ and $6$ be $1\to 6$, $1\to 4$, $6\to 3$, $6\to 1$. If the direction $4\to 1$ lies in the cyclic order between $1\to 6$ and $6\to 1$ then the paths $14$ and $16$ give a Vassiliev's obstruction (see fig.~\ref{fig:case_K33_v6_1}.4). Analogously, we get an obstruction if the direction $3\to 6$ lies between $1\to 6$ and $6\to 1$. Hence, suppose that the directions $4\to 1$ and $3\to 6$ lie between $1\to 4$ and $6\to 6$ (see fig.~\ref{fig:case_K33_v6_1}.5). Then the paths $14$ and $36$ give a Vassiliev's obstruction.

  \begin{figure}[!ht]
  \centering
  \begin{tabular}{ccc}
  \includegraphics[width=0.25\textwidth]{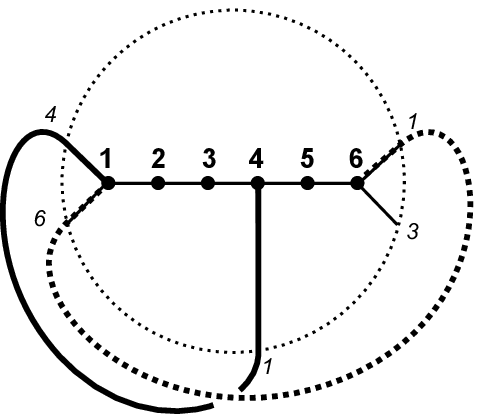} &
  \includegraphics[width=0.22\textwidth]{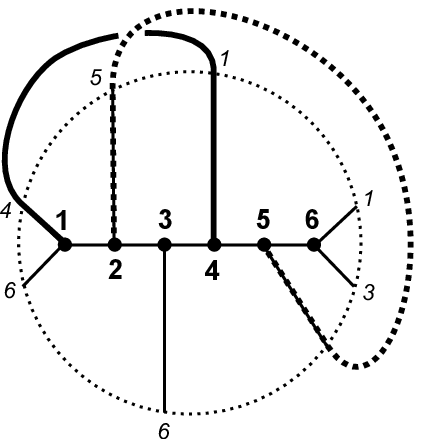} &
  \includegraphics[width=0.25\textwidth]{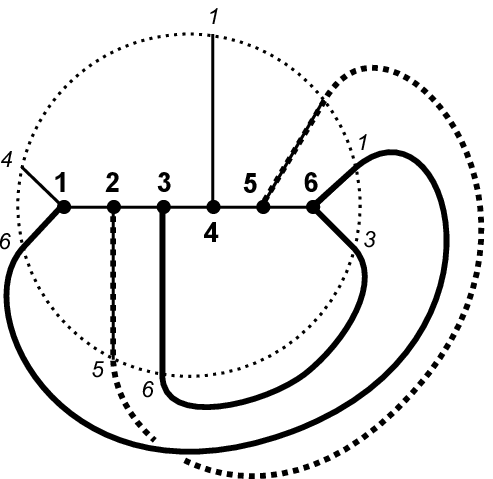}\\
  1) & 2) & 3)
  \end{tabular}\\
    \begin{tabular}{cc}
  \includegraphics[width=0.25\textwidth]{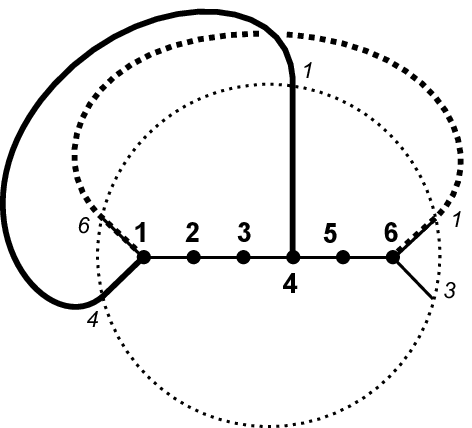} &
  \includegraphics[width=0.21\textwidth]{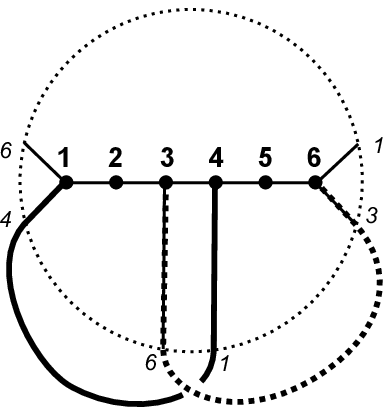}\\
  4) & 5)
  \end{tabular}
  \caption{Case $K_{3,3}$.e.1}\label{fig:case_K33_v6_1}
 \end{figure}

Case $K_{3,3}$.e.2. The graph $\bar\Gamma^{int}_W$ is a tree with one vertex of degree $3$.\\
Enumerate the vertices of $\bar\Gamma^{int}_W$ as shown in fig.~\ref{fig:node_graph_K33_v6}.2).

Let the cyclic order of the directions be as follows: $2\to 3$, $2\to 5$, directions of the vertex $3$, directions of the vertex $5$ (see fig.~\ref{fig:case_K33_v6_2}.1). Then the paths $23$ and $25$ give a Vassiliev's obstruction. Hence, later we can suppose that the order is: $2\to 5$, $2\to 3$, directions of the vertex $3$, directions of the vertex $5$.

Let the direction $3\to 6$ follow the direction $2\to 3$ in the cyclic order (see fig.~\ref{fig:case_K33_v6_2}.2). Then a Vassiliev's obstruction can be obtained from the paths $23$ and $36$. Hence, hereinafter we can suppose that the direction $3\to 2$ follows $2\to 3$ and (symmetrically) $2\to 5$ follows $5\to 2$. Thus, the cyclic order of the directions of the vertices $2$, $3$, $5$ will be $2\to 5$, $2\to 3$, $3\to 2$, $3\to 6$, $5\to 4$, $5\to 2$.

Look at the position of the direction $4\to 5$. If it lies between the directions $2\to3$ and $3\to 2$ (see fig.~\ref{fig:case_K33_v6_2}.3) then we get a Vassiliev's obstruction from the paths $23$ and $45$. Analogously, an obstruction appears when the direction $6\to 3$ lies between $2\to 5$ and $5\to 2$.

Thus, we can suppose that $4\to 5$ and $6\to 3$ lies between $3\to 6$ and $5\to 4$ (see fig.~\ref{fig:case_K33_v6_2}.4). In this case we obtain an obstruction if we take the paths $36$ and $45$.

  \begin{figure}
  \centering
  \begin{tabular}{ccc}
  \includegraphics[width=0.25\textwidth]{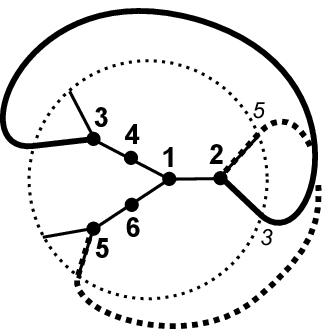} &
  \includegraphics[width=0.25\textwidth]{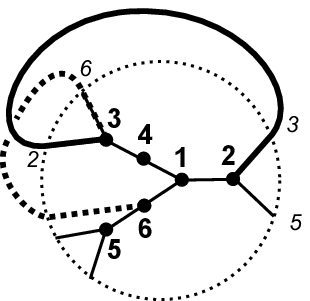}\\
  1) & 2)
  \end{tabular}\\
    \begin{tabular}{cc}
  \includegraphics[width=0.25\textwidth]{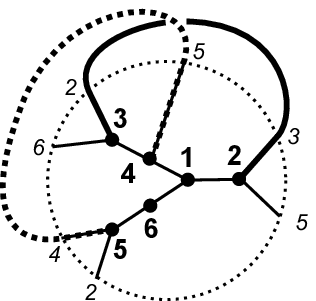} &
  \includegraphics[width=0.27\textwidth]{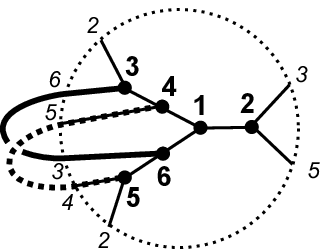}\\
  3) & 4)
  \end{tabular}
  \caption{Case $K_{3,3}$.e.2}\label{fig:case_K33_v6_2}
 \end{figure}

Case $K_{3,3}$.e.3. The graph $\bar\Gamma^{int}_W$ is a tree with two vertices of degree $3$.\\
Enumerate the vertices of $\bar\Gamma^{int}_W$ as shown in fig.~\ref{fig:node_graph_K33_v6}.3). Without loss of generality we can also suppose that the cyclic order of the directions of the vertices $3,4,5,6$ is the following: directions of the vertex $3$, directions of $4$, directions of $6$, directions of $5$ (see fig.~\ref{fig:case_K33_v6_3}). Then the paths $36$ and $45$ give a Vassiliev's obstruction.

 \begin{figure}[!ht]
  \centering\includegraphics[width=0.25\textwidth]{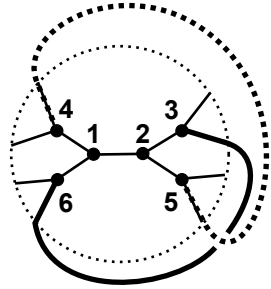}
  \caption{Case $K_{3,3}$.e.3}\label{fig:case_K33_v6_3}
 \end{figure}

Case $K_{3,3}$.e.4. The graph $\bar\Gamma^{int}_W$ contains one cycle of length $4$.\\
Enumerate the vertices of $\bar\Gamma^{int}_W$ as shown in fig.~\ref{fig:node_graph_K33_v6}.4). Then the projections of the paths $36$ and $45$ form a Vassiliev's obstruction in $G$ (see fig.~\ref{fig:case_K33_v6_4}).

 \begin{figure}[!ht]
  \centering\includegraphics[width=0.25\textwidth]{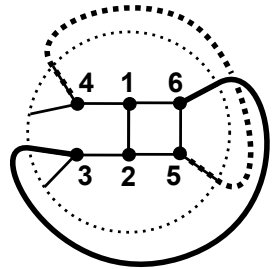}
  \caption{Case $K_{3,3}$.e.4}\label{fig:case_K33_v6_4}
 \end{figure}

Case $K_{3,3}$.e.5. The graph $\bar\Gamma^{int}_W$ contains two cycles of length $4$.\\
Enumerate the vertices of $\bar\Gamma^{int}_W$ as shown in fig.~\ref{fig:node_graph_K33_v6}.5). Then we obtain a Vassiliev's obstruction from the paths $36$ and $45$ form a Vassiliev's obstruction in $G$ (see fig.~\ref{fig:case_K33_v6_5}).

 \begin{figure}[!ht]
  \centering\includegraphics[width=0.25\textwidth]{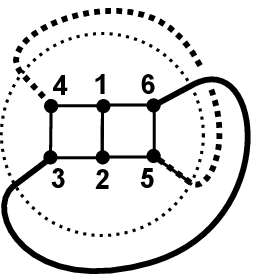}
  \caption{Case $K_{3,3}$.e.5}\label{fig:case_K33_v6_5}
 \end{figure}

Case $K_{3,3}$.e.6. The graph $\bar\Gamma^{int}_W$ is a cycle of length $4$.\\
Enumerate the vertices of $\bar\Gamma^{int}_W$ as shown in fig.~\ref{fig:node_graph_K33_v6}.6). Then a Vassiliev's obstruction in $G$ is given by the paths $14$ and $25$ (see fig.~\ref{fig:case_K33_v6_6}).

 \begin{figure}[!ht]
  \centering\includegraphics[width=0.25\textwidth]{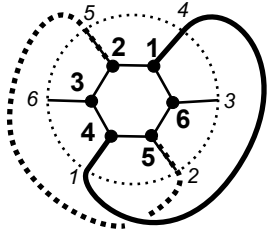}
  \caption{Case $K_{3,3}$.e.6}\label{fig:case_K33_v6_6}
 \end{figure}

Thus, the case $K_{3,3}$.e is finished.

Case $K_{3,3}$.f. Any web in $\wG$ contains no more that one vertex of $\Gamma$. This case should be treated in a different way than the previous cases.

Let $\hG=\pi(\Gamma)$ be the projection of the graph $\Gamma$ into $G$. With some abuse of notation, we denote the projections of the vertices $\Gamma$ as $i, i=1,\dots,6$, and  the projections of the edges of $\Gamma$ as $ij, 1\le i<j\le 6$. By assumptions of the case, the projections of the vertices are all different. By lemma~\ref{lem:transversal_path_projection}, the paths  $ij, 1\le i<j\le 6$, in $\hG$ don't intersect each other transversally. The degree of the vertices $i, i=1,\dots,6$, of $\hG$ are odd and the degree of any other vertex in $\hG$ is even.

Let $G'$ be the graph obtained from $G$ by deleting all the edges of the graph $\hG$. Since $G$ is even, the vertices $i, i=1,\dots,6$, are the only vertices in $G'$ of odd degree. Denote $G'_1$ to be the union of components of $G'$ which contain The vertices $1$, $3$ or $5$. If $G'$ contains neither the vertex $2$, $4$ nor $6$ then there are exactly three odd vertices in $G'_1$ that can not be true. Hence, $G'_1$ contains a vertex $2k, k=1,2$ or $3$, therefore, there is a path $\gamma$ in $G'$ which connects a vertex $2l-1, l=1,2$ or $3$. Without loss of generality, we can assume that $l=1$. The path $\gamma$ is a path in $G$ which has no common edges with the paths $ij, 1\le i<j\le 6$. We can suppose $\gamma$ has no transversal self-intersections.

\begin{nn_remark}
The proof of existence of the path $\gamma$ is the only place in the proof of the main theorem where we use the fact that $G$ is even.
\end{nn_remark}

Case $K_{3,3}$.f.1. The path $\gamma$ intersects transversally some path $ij, 1\le i<j\le 6$.\\
Let $P$ be the transversal intersection of $\gamma$ with paths $ij$ which is the nearest to the vertex $1$ on $\gamma$. If at some vertex of $G$ the path $\gamma$ intersects with paths $ij$ several times, the proximity of transversal intersections to the vertex $1$ can be determined by the vertex chord diagram. Note that the order is well defined, since there is no transversal intersections in $\hG$ and the chord, which correspond to paths $ij$ of $\hG$, don't intersect each other. Denote $\gamma_0$ to be the part of $\gamma$ between $1$ and $P$.

According to the path $ij$ the transversal intersection $P$ belong, there are two possible cases.

1. The intersection point $P$ belongs to path $1j, j=2,4$ or $6$. We will suppose $j=2$ (see fig.~\ref{fig:K33_v1_transversal_path reduction}). Let $Q$ be the last transversal intersection of $\gamma$ with the part $1P$ of the path $12$ ($Q$ can coincide with $P$) and $\gamma_1$ be the rest part of $\gamma$ from the point $Q$.Then we can perform the following reduction: we replace the path $12$ with the path $\gamma_0P2$ and replace the path $\gamma$ with the path $1Q\gamma_1$. Since $P$ is the first transversal intersection, the path $\gamma_0P2$ has no transversal intersections with other paths in $\hG$. On the other hand, since $Q$ is the last transversal intersection, the path $1Q\gamma_1$ has no transversal self-intersections and the number of transversal intersections of the path $1Q\gamma_1$ with paths in modified graph $\hG$ is less that the number of transversal intersections in the original path $\gamma$.

  \begin{figure}[!ht]
  \centering\includegraphics[width=0.7\textwidth]{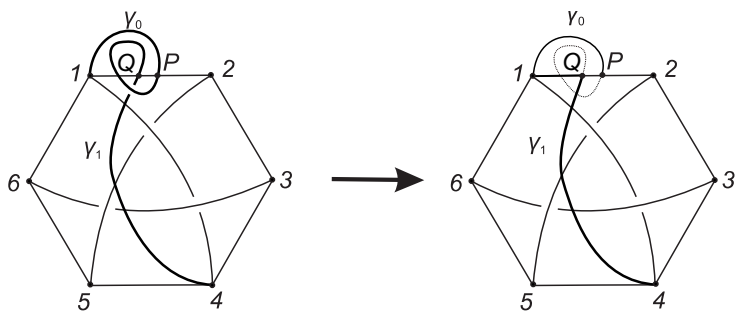}
  \caption{Case $K_{3,3}$.f.1: reduction of the path $\gamma$}\label{fig:K33_v1_transversal_path reduction}
 \end{figure}

2. The intersection point $P$ belongs to path $ij$ such that $i,j>1$. We can assume that $i=3$, so $j=2,4$ or $6$. Consider the restriction of the cyclic order at the vertex $1$ to the paths $12$, $14$, $16$ and $\gamma$. If $\gamma$ and $1j$ are not adjacent in the restricted cyclic order (see fig.~\ref{fig:K33_v1_transversal_path_cycles} left) then the cycles $\gamma_0P341$ and $12561$ form a Vassiliev's obstruciton (we assume here $j=2$). If $\gamma$ and $1j$ are adjacent in the restricted cyclic order (see fig.~\ref{fig:K33_v1_transversal_path_cycles} right) then the cycles $\gamma_0P41$ and $12561$ form a Vassiliev's obstruciton (we assume here $j=4$).

  \begin{figure}[!ht]
  \centering\includegraphics[width=0.25\textwidth]{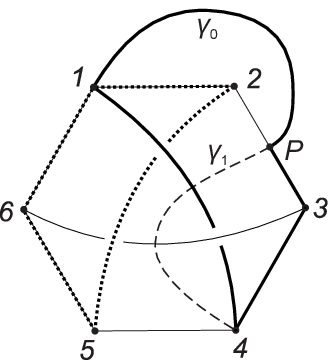}
  \quad \includegraphics[width=0.25\textwidth]{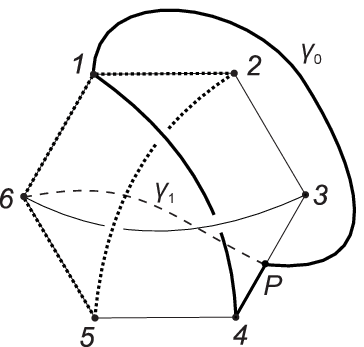}
  \caption{Case $K_{3,3}$.f.1: transversal cycles}\label{fig:K33_v1_transversal_path_cycles}
 \end{figure}

Case $K_{3,3}$.f.2. The path $\gamma$ does not have transversal intersections with paths $ij, 1\le i<j\le 6$.\\
We can suppose that the ends of $\gamma$ are $1$ and $2$. Consider the cyclic order at the vertex $1$ restricted to the paths $12$, $14$, $16$ and $\gamma$ and the cyclic order at the vertex $2$ restricted to the paths $21$, $23$, $25$ and $\gamma$.

 Let the paths $\gamma$ and $12$ be not adjacent in the cyclic order either at the vertex $1$ or at $2$ (for instance, at $1$), see fig.~\ref{fig:case_K33_v1} left. Then the cycles $\gamma21$ and $14561$ have one transversal intersection at the vertex $1$ and form a Vassiliev's obstruciton.

  Let the paths $\gamma$ and $12$ be adjacent in the cyclic orders at the vertex $1$ and at $2$ (see fig.~\ref{fig:case_K33_v1} right). Then the cycles $\gamma2561$ and $12341$ have a transversal intersection at the vertex $2$.

  \begin{figure}[!ht]
  \centering\includegraphics[width=0.25\textwidth]{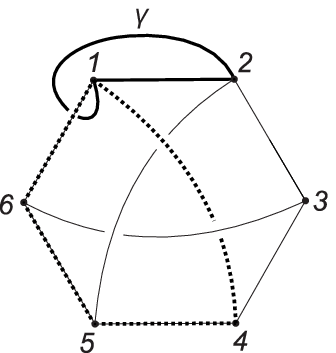}
  \quad \includegraphics[width=0.25\textwidth]{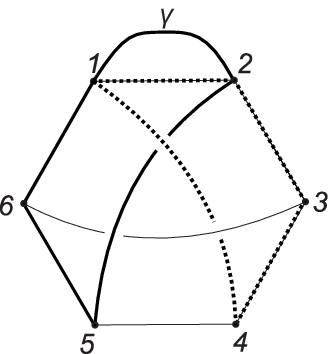}
  \caption{Case $K_{3,3}$.f.2: $\gamma$ has no transversal intersections}\label{fig:case_K33_v1}
 \end{figure}

Thus, theorem~\ref{thm:main_theorem} is proved.


\noindent{\bf Proof of Theorem~\ref{thm:main_theorem2}.}
According to the proof of theorem~\ref{thm:main_theorem} we need only consider the case
when the graph $K_{3,3}$ can be immersed into the \st-graph $G$ such that it has no
transversal self-intersection points (case $K_{3,3}$.f in the roof above) and sends the
vertices of $K_{3,3}$ to different vertices of $G$. Let $\Gamma$ be the image of the
immersion. Denote the images of the vertices of $K_{3,3}$ as $1,2,3,4,5,6$ so that the
vertices of different parity corresponds to adjacent vertices of $K_{3,3}$. Let $ij$
denote the image of the edge which connects the vertices $i$ and $j$ in $K_{3,3}$.

If the subgraph $\Gamma$ has no self-intersections then it is homeomorphic to $K_{3,3}$.
So, let us assume $\Gamma$ has a self-intersection point $P$. We shall show below that
either $\Gamma$ contains a pair of transversal cycles or it can be reduced to an
immersion $K_{3,3}$ that contains less number of
edges of the graph $G$.

Case a. There is a vertex of $\Gamma$ at the intersection point $P$.\\
Denote $a$ to be the vertex at $P$ and let $i,j,k$ be the adjacent vertices to $a$. Since $P$ is an intersection point there is an edge $pq$ that goes through $P$. If $P$ is a multiple intersection we take the edge whose arc in the vertex chord diagram is not separated from the vertex $a$ (see fig.~\ref{fig:graph_K33_reduction_chord_diagram_v}).

  \begin{figure}[!ht]
  \centering\includegraphics[width=0.2\textwidth]{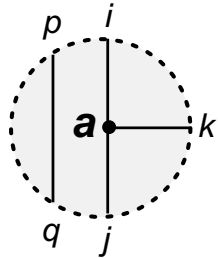}
  \caption{Case a: there is a vertex at the intersection point}\label{fig:graph_K33_reduction_chord_diagram_v}
 \end{figure}

We can suppose that $a=1$ and $\{i,j,k\}=\{2,4,6\}$ where the direction $14$ is separated from the arc $pq$ with the directions $12$ and $16$ in the cyclic order at the point $P$.

Case a.1. $p=1$ or $q=1$.\\
In this case we can reduce the subgraph $\Gamma$. Assume that $p=1$. The point $P$ splits the path $1q$ into paths $\gamma_1$ and $\gamma_2$. Then we can replace the path $1q$ with the path $\gamma_2$ and get an immersion of the graph $K_{3,3}$ which has less number of edges of $G$ (see fig.~\ref{fig:graph_K33_reduction_v_1i}).

  \begin{figure}[!ht]
  \centering\includegraphics[width=0.5\textwidth]{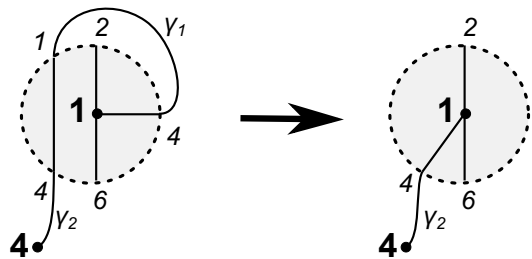}
  \caption{Case a.1: $1\in\{p,q\}$}\label{fig:graph_K33_reduction_v_1i}
 \end{figure}

Case a.2. The edge $pq$ is not incident to the vertex $1$ and $p,q\ne 4$.\\
We can assume that $\{p,q\}=\{2,3\}$. The point $P$ splits the edge $pq$ into subpaths $\gamma_1$ and $\gamma_2$.

If the direction of the subpath incident to the vertex $2$ follows the direction $12$ in the cyclic order at $P$ (see fig.~\ref{fig:graph_K33_reduction_v_23} left) then the cycles $143\gamma_2$ and $1652\gamma_1$ have a unique transversal intersection at the point $P$.

If the direction of the subpath incident to the vertex $2$ does not follow the direction $12$ in the cyclic order at $P$ (see fig.~\ref{fig:graph_K33_reduction_v_23} right) then the cycles $12\gamma_2$ and $163\gamma_1$ form a Vassiliev's obstruction.

  \begin{figure}[!ht]
  \centering\includegraphics[width=0.27\textwidth]{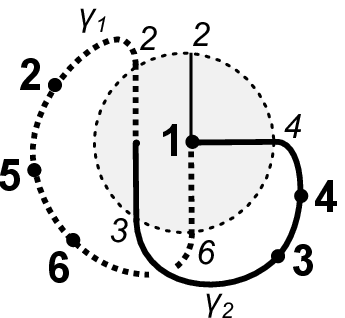}\qquad
  \includegraphics[width=0.25\textwidth]{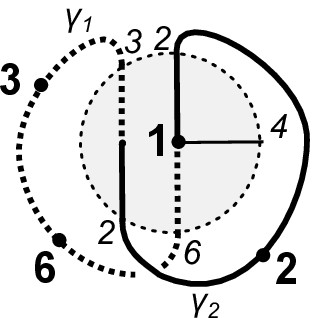}
  \caption{Case a.2: vertex $4$ is not incident to the edge $pq$}\label{fig:graph_K33_reduction_v_23}
 \end{figure}

Case a.3. The edge $pq$ is not incident to the vertex $1$ and $\{p,q\}\ni 4$.\\
We can assume that $\{p,q\}=\{3,4\}$. The point $P$ splits the edge $34$ into subpaths $\gamma_1$ and $\gamma_2$. Without loss of generality we can assume the direction of the subpath incident to the vertex $3$ follows the direction $12$ in the cyclic order at $P$ (see fig.~\ref{fig:graph_K33_reduction_v_34}.Then the cycles $14\gamma_2$ and $163\gamma_1$ are transversal.

  \begin{figure}[!ht]
  \centering\includegraphics[width=0.25\textwidth]{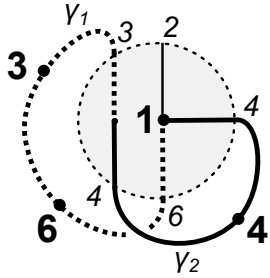}
  \caption{Case a.3: vertex $4$ is incident to the edge $pq$}\label{fig:graph_K33_reduction_v_34}
 \end{figure}

Case b. There is no vertices of $\Gamma$ at the intersection point $P$.\\
Since $P$ is an intersection point there are edges $ij$ and $kl$ which go through $P$. Without loss of generality, we can assume that the arcs, which correspond to the edges in the vertex chord diagram, are not separated from each other (see fig.~\ref{fig:graph_K33_reduction_chord_diagram}). We suppose that $i=1$, $j=2$

  \begin{figure}[!ht]
  \centering\includegraphics[width=0.15\textwidth]{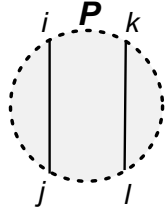}
  \caption{Case b: there is no vertices at the intersection point}\label{fig:graph_K33_reduction_chord_diagram}
 \end{figure}

Case b.1. The edges $ij$ and $kl$ coincide, i.e. $\{i,j\}=\{k,l\}=\{1,2\}$.\\
There are two possibilities.

Case b.1.1. The arcs $ij$ and $kl$ have different orientations ($k=2,l=1$).
The arcs split the edge $12$ into three paths $\gamma_1,\gamma_2,\gamma_3$ where $\gamma_1$ contains the vertex $1$ and $\gamma_3$ contains the vertex $3$. Then we can replace the path $12$ with the path $\gamma_1\cup\gamma_3$ and get an immersion of the graph $K_{3,3}$ which has less number of edges of $G$ (see fig.~\ref{fig:graph_K33_reduction_1212}).

  \begin{figure}[!ht]
  \centering\includegraphics[width=0.7\textwidth]{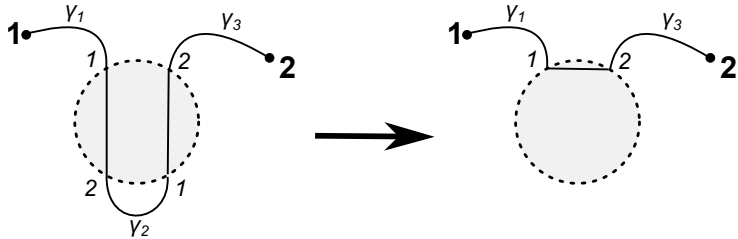}
  \caption{Case b.1.1: arcs belong to one edge and have opposite orientations}\label{fig:graph_K33_reduction_1212}
 \end{figure}

Case b.1.2. The arcs $ij$ and $kl$ have the same orientation ($k=2,l=1$).
The point $P$ splits the edge $12$ into paths $\gamma_1,\gamma_2,\gamma_3$. Then the cycles $\gamma_2$ and $\gamma_1\gamma_3 2341$ form a Vassiliev's obstruction (see fig.~\ref{fig:graph_K33_reduction_1221}).

  \begin{figure}[!ht]
  \centering\includegraphics[width=0.25\textwidth]{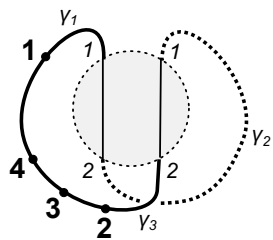}
  \caption{Case b.1.2: arcs belong to one edge and have the same orientation}\label{fig:graph_K33_reduction_1221}
 \end{figure}

Case b.2. The edges $ij$ and $kl$ are adjacent, i.e. $\{i,j\}\ne\{k,l\}$ and $\{i,j\}\cap\{k,l\}\ne\emptyset$.\\
We can suppose that $\{k,l\}=\{1,4\}$. The point $P$ splits the edge $12$ into subpaths $\gamma_1$ and $\gamma_1'$ and splits the edge $14$ into subpaths $\gamma_2$ and $\gamma_2'$, where $\gamma_1$ and $\gamma_2$ contain the vertex $1$.

Case b.2.1. The arcs $ij$ and $kl$ have different orientations ($k=4,l=1$).
Then the cycles $\gamma_1\gamma_2$ and $\gamma_1'254\gamma_2'$ have a unique transversal intersection at the point $P$ (see fig.~\ref{fig:graph_K33_reduction_1214}).

  \begin{figure}[!ht]
  \centering\includegraphics[width=0.25\textwidth]{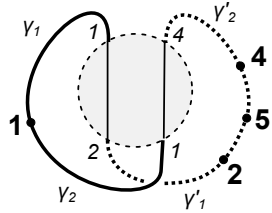}
  \caption{Case b.2.1: arcs belong to adjacent edges and have opposite orientations}\label{fig:graph_K33_reduction_1214}
 \end{figure}

Case b.2.2. The arcs $ij$ and $kl$ have different orientations ($k=1,l=4$).
In this case we can move the vertex $1$ to the point $P$ and replace the edge $12$ with the path $\gamma_1'$, the edge $14$ with the path $\gamma_2'$ and the edge $16$ with the path $\gamma_1 16$ (see fig.~\ref{fig:graph_K33_reduction_1241}). Thus we get a new immersion of the graph $K_{3,3}$ which has less number of edges of $G$.

  \begin{figure}[!ht]
  \centering\includegraphics[width=0.7\textwidth]{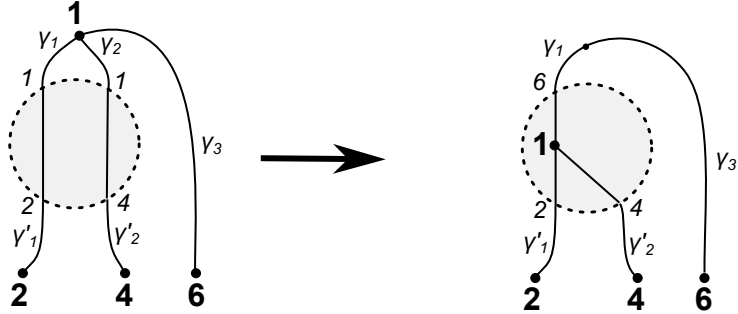}
  \caption{Case b.2.2: arcs belong to adjacent edges and have the same orientation}\label{fig:graph_K33_reduction_1241}
 \end{figure}

Case b.3. The edges $ij$ and $kl$ are not adjacent, i.e. $\{i,j\}\cap\{k,l\}=\emptyset$.\\
We can suppose that $\{k,l\}=\{3,4\}$. The point $P$ splits the edge $12$ into subpaths $\gamma_1$ and $\gamma_1'$ and splits the edge $34$ into subpaths $\gamma_2$ and $\gamma_2'$, where $\gamma_1$ contains the vertex $1$ and $\gamma_2$ contains the vertex $3$.

Case b.3.1. The arcs $12$ and $34$ have different orientations ($k=4,l=3$).
In this case the cycles $\gamma_1\gamma_2 361$ and $\gamma_1'\gamma_2' 452$ have a unique transversal intersection at the point $P$ (see fig.~\ref{fig:graph_K33_reduction_1234} left).

Case b.3.2. The arcs $12$ and $34$ have the same orientation ($k=3,l=4$).
Then the cycles $\gamma_1\gamma_2' 41$ and $\gamma_1'\gamma_2 32$ form a Vassiliev's obstruction (see fig.~\ref{fig:graph_K33_reduction_1234} right).

  \begin{figure}[!ht]
  \centering\includegraphics[width=0.28\textwidth]{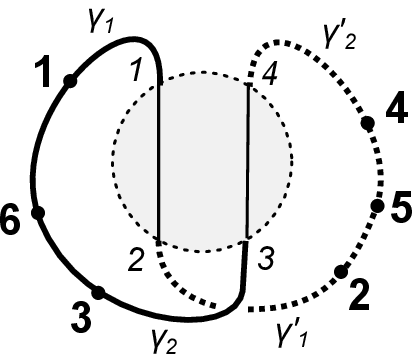}\qquad
  \includegraphics[width=0.25\textwidth]{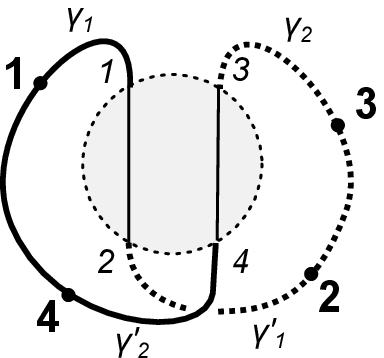}
  \caption{Case b.3: arcs belong to nonadjacent edges}\label{fig:graph_K33_reduction_1234}
 \end{figure}

Thus, if the graph $\Gamma$ contains a self-intersecton point, then
either $\Gamma$ contains a pair of transversal cycles or it can be reduced to an
immersion $K_{3,3}$ which contains less number of
edges of the graph $G$. Therefore, any immersion of $K_{3,3}$ into $G$ without transversal self-intersections contains a pair of transversal cycles or can be reduced to an embedding of $K_{3,3}$.

\section*{Acknowledgments}
The author is grateful to V.\,O.~Manturov and  D.\,P.~Ilyutko for their interest to this work. The author was partially supported by grants RFBR 13-01-00830-a and 14-01-31288-mol-a, and grant of RF President NSh -- 1410.2012.1.

\end{document}